\documentclass[12pt]{article}
\usepackage{amsmath}
\usepackage{amsfonts}
\usepackage{amssymb}
\usepackage{amscd}
\usepackage{amsthm}
\usepackage{color}

\input xypic

\newcommand{\Fc}{{\mathcal F}}

\newcommand{\Gc}{{\mathcal G}}

\newcommand{\A}{{\mathbf A}}
\newcommand{\Ab}{{\mathbb A}}
\renewcommand{\a}{\mathfrak{a}}
\newcommand{\m}{\mathfrak{m}}
\newcommand{\p}{\mathfrak{p}}

\newcommand{\OO}{{\mathcal O}}

\newcommand{\codim}{{\rm codim}}

\newcommand{\Spec}{{\rm Spec}}

\theoremstyle{plain}
\newtheorem{theor}{Theorem}
\newtheorem{prop}[theor]{Proposition}

\newtheorem{corol}[theor]{Corollary}
\newtheorem{lemma}[theor]{Lemma}

\theoremstyle{remark}
\newtheorem{rmk}[theor]{Remark}

\theoremstyle{definition}

\newtheorem{defin-prop}[theor]{Definition-Proposition}

\title{Intersections of adelic groups on a surface}
\author{Roman Budylin and Sergey Gorchinskiy\\ \\
\small{Steklov Mathematical Institute, Moscow, Russia}\\
\small{e-mail: {\tt budylin@mi.ras.ru}}\\
\small{e-mail: {\tt gorchins@mi.ras.ru}}}

\date{}

\begin{document}
\maketitle

\begin{abstract}
We solve a technical problem related to adeles on an algebraic surface. Given a finite set of natural numbers up to two, one associates an adelic group. We show that this operation commutes with taking intersections if the surface is defined over an uncountable field and we provide a counterexample otherwise.

\end{abstract}

\section{Introduction}

Adeles for surfaces were introduced by A.\,N.\,Parshin~\cite{Par} as a generalization of classical adeles for global fields, in particular, fields of rational functions on curves (over finite fields). In this generalization, points on curves are replaced by flags, that is, chains of embedded irreducible subvarieties. One defines several adelic groups according to codimensions of members in flags. More precisely, given a surface $X$ and a subset \mbox{$I\subset \{0,1,2\}$}, one associates an adelic group $\A_X(I,\OO_X)$ (this is a particular case of an adelic group $\A_X(I,\Fc)$ defined for an arbitrary quasicoherent sheaf $\Fc$ on $X$). If $I\subset J$, then there is a canonical map $\varphi_{IJ}\colon\A_X(I,\OO_X)\to \A_X(J,\OO_X)$, which is injective if $X$ is regular, Proposition~\ref{prop:inj}. Thus all adelic groups are subgroups in the biggest one, $\A_X(\{0,1,2\},\OO_X)$, and a natural question is whether \mbox{$\A_X(I\cap J,\OO_X)$} is equal to $\A_X(I,\OO_X)\cap \A_X(J,\OO_X)$. Note that the analogous question for rational, also called uncomplete, adeles is trivial,~\cite[\S2]{Par}.

When $X$ is projective, a positive answer to this question was obtained in~\cite[Prop. 4.3]{FP} by means of global methods. This was used later in~\cite{PO} in order to prove Riemann--Roch theorem for surfaces using adeles. In this paper we give a positive answer to the above question when the ground field is uncountable, Theorem~\ref{theor-main}, and provide a counterexample in the affine case when the ground field is countable, Theorem~\ref{theor:contr}. Also, we give a positive answer to the above question in the projective case for an arbitrary locally free sheaf of finite rank. The proof is essentially the same as the proof of~\cite[Prop. 4.3]{FP} for the structure sheaf, we include it here with a kind permission of A.\,N.\,Parshin.

Generalization to a higher-dimensional case remains open. Actually, it is not even known whether the map $\varphi_{IJ}$ is injective (even for regular varieties). Notice that analogues of all statements of the paper for restricted adeles on projective Cohen--Macauley varieties of arbitrary dimension hold true by~\cite[Theor. 4,~5]{Osi}.

We are very grateful to D.\,Osipov for many useful discussions and to A.\,N.\,Parshin for encouragement and important comments. We highly appreciate excellent working conditions in the Hausdorff Research Institute for Mathematics (HIM), Bonn, where the paper was initiated. This work was partially supported by the grants RFBR 11-01-00145, NSh-5139.2012.1, AG Laboratory NRU HSE, RF government grant, ag. 11.G34.31.0023. The work of S.\,G. was partially supported by RFBR grants 12-01-31506, 12-01-3302 and by Dmitry Zimin's Dynasty Foundation.

\section{Statement of the main result}\label{sec:prel}

First we recall some general facts about adeles on Noetherian schemes. However, as our main result concerns surfaces only, the reader may easily restrict himself by this from the very beginning. Notice that the definition of adeles on a surface has a much more explicit version,~\cite{Par},~\cite{Osi2},~\cite{Mor}.

Let $X$ be a Noetherian scheme and $\Fc$ be a quasicoherent sheaf on~$X$. By $S(X)_p$, $p\geqslant 0$, denote the set of all non-degenerate length $p$ flags on $X$, that is, sequences of schematic points $(\eta_0,\ldots,\eta_p)$ such that $\eta_{i+1}\in \bar\eta_i$ and $\eta_{i+1}\ne\eta_i$, where $\bar\eta$ denotes the closure of a point~$\eta$ in $X$. Given a subset $S\subset S(X)_p$, by $\A_X(S,\Fc)$ denote the corresponding group of adeles,~\cite{Bei},~\cite{Hub}. The functor $\Fc\mapsto \A_X(S,\Fc)$ is exact and commutes with filtered colimits. There is a canonical embedding,~\cite[Prop. 2.1.4]{Hub},
$$
\mbox{$\A_X(S,\Fc)\hookrightarrow \prod\limits_{\Delta\in S}\A_X(\{\Delta\},\Fc)$}\,.
$$
Thus an adele $a\in \A_X(S,\Fc)$ is uniquely determined by its local components \mbox{$a_{\Delta}\in \A_X(\{\Delta\},\Fc)$}. By definition, put $\A_X(\varnothing,\Fc):=\Fc(X)=H^0(X,\Fc)$.
We will use the following facts.

\begin{lemma}\label{prop-finitemorph}
Let $f\colon X\to Y$ be a finite morphism between Noetherian schemes. Then for any subset $T\subset S(Y)_p$, $p\geqslant 0$, there is a canonical isomorphism
$$
\A_Y(T,f_*\Fc)\stackrel{\sim}\longrightarrow\A_X(f^{-1}T,\Fc)\,,
$$
where $f^{-1}T\subset S(X)_p$ is defined in a natural way.
\end{lemma}

The proof of Lemma~\ref{prop-finitemorph} is completely analogous to the proof of~\cite[Prop. 3.1.7]{Yek}.

\begin{lemma}\label{lemma-Denis}
If $X$ is affine, then for any subset $S\subset S(X)_p$, $p\geqslant 0$, there is a canonical isomorphism

$$
\A_X(S,\OO_X)\otimes_{\OO_X(X)}\Fc(X)\stackrel{\sim}\longrightarrow \A_X(S,\Fc)\,.
$$

\end{lemma}

The proof of Lemma~\ref{lemma-Denis} is similar to the proof of~\cite[Prop. 1.5]{Osi}. Namely, since the functor $\Fc\mapsto \A_X(S,\Fc)$ commutes with filtered colimits, one may assume that $\Fc$ is coherent. Further, using the exactness of this functor and finite presentation of coherent sheaves, one reduces the problem to the obvious case $\Fc=\OO_X$.

\medskip

In what follows we assume that $X$ is irreducible. Let $d$ be the dimension of $X$. Let~$I$ be a subset in $\{0,1,\ldots,d\}$, that is, $I=\{i_0,\ldots,i_p\}$ for a strictly increasing sequence of integers $0\leqslant i_0<\ldots<i_p\leqslant d$, $0\leqslant p\leqslant d$. Put
$$
\A_X(I,\Fc):=\A_X(S(I),\Fc)\,,
$$
where $S(I)$ is the set of all flags $(\eta_0,\ldots,\eta_p)$ with $\codim_X(\bar\eta_j)=i_j$, $0\leqslant j\leqslant p$. In particular, $\A_X(\{0\},\Fc)$ is the fiber of $\Fc$ at the generic point of $X$. By $\underline{\A}_X(I,\Fc)$, $I\ne \varnothing$, denote a flabby sheaf on $X$ defined by the formula
$$
\underline{\A}_X(I,\Fc)(U):=\A_U(I,\Fc|_U)\,,
$$
where $U$ is an open subset in $X$. By definition, put $\underline{\A}_X(\varnothing,\Fc):=\Fc$.
Lemma~\ref{lemma-Denis} implies immediately the following fact.

\begin{corol}\label{cor-Denis}
For any subset $I\subset \{0,\ldots,d\}$, there is a canonical isomorphism of sheaves
$$
\underline{\A}_X(I,\OO_X)\otimes_{\OO_X}\Fc\stackrel{\sim}\longrightarrow \underline{\A}_X(I,\Fc)\,.
$$
\end{corol}

Given subsets $I\subset J\subset \{0,\ldots,d\}$, there is a canonical map induced by boundary maps on adelic groups,~\cite[Sec. 2.2]{Hub},
$$
\varphi_{IJ}\colon\A_X(I,\Fc)\to \A_X(J,\Fc)\,.
$$

\begin{prop}\label{prop:inj}
If $X$ is a regular irreducible surface over a field and $\Fc$ is flat, then for all subsets \mbox{$I\subset J\subset \{0,1,2\}$}, the map $\varphi_{IJ}$ is injective.
\end{prop}
\begin{proof}
The explicit definition of adeles on a regular surface,~\cite[Sec. 2]{Par},~\cite[Sec. 3.3]{Osi2},~\cite[Sec. 8.5]{Mor}, implies the proposition for $\Fc=\OO_X$. Furthermore, the morphism of sheaves
$$
\underline{\A}_X(I,\OO_X)\to\underline{\A}_X(J,\OO_X)
$$
induced by $\varphi_{IJ}$ is injective as well. Combining flatness of $\Fc$, Corollary~\ref{cor-Denis}, and left exactness of $H^0$, we obtain the required statement.
\end{proof}

\begin{rmk}
It seems that Proposition~\ref{prop:inj} is also true when $X$ is a normal excellent two-dimensional irreducible Noetherian scheme.
\end{rmk}

Thus under the conditions of Proposition~\ref{prop:inj} all groups $\A_X(I,\Fc)$ are canonically embedded into the group $\A_X(\{0,1,2\},\Fc)$. The main question that we address in the paper is whether \mbox{$\A_X(I\cap J,\Fc)$} is equal to $\A_X(I,\Fc)\cap \A_X(J,\Fc)$ for arbitrary subsets $I$ and $J$ in~$\{0,1,2\}$.

\medskip

Recall that a flat coherent sheaf is the same as a locally free sheaf of finite rank, \cite[Prop.(3.G)]{M}. Our main result is as follows.

\begin{theor}\label{theor-main}
Let $X$ be a regular irreducible surface over a field $k$, $\Fc$ be a flat quasicoherent sheaf on $X$, and let $I$ and $J$ be two subsets in $\{0,1,2\}$. Suppose that one of the following conditions is satisfied:
\begin{itemize}
\item[(i)]
$I\cap J=I\smallsetminus\{0\}$,
\item[(ii)]
the field $k$ is uncountable,
\item[(iii)]
$\Fc$ is locally free of finite rank and $X$ is projective.
\end{itemize}
Then there is an equality
$$
\A_X(I\cap J,\Fc)=\A_X(I,\Fc)\cap \A_X(J,\Fc)\,,
$$
where the intersection is taken in $\A_X(\{0,1,2\},\Fc)$.
\end{theor}

The proof of the theorem under condition~$(i)$ is simple and is given in Section~\ref{sect:i}. The proofs of the theorem under conditions~$(ii)$ and $(iii)$ are based on several auxiliary statements and are given in Sections~\ref{sect:ii} and~\ref{sect:iii}, respectively. Note that the proof of the theorem under condition~$(ii)$ uses only elementary facts from commutative algebra, while the proof of the theorem under condition~$(iii)$ uses Serre duality. In particular, for a projective regular surface over an uncountable field one has two different proofs.

\section{Condition~$(i)$ case}\label{sect:i}

Let $X$ and $\Fc$ be as in Theorem~\ref{theor-main}.

\begin{lemma}\label{lemma-red}
Suppose that for any affine open subset $U\subset X$, we have
$$
\A_{U}(I\cap J,\OO_{U})=\A_{U}(I,\OO_{U})\cap \A_{U_i}(J,\OO_{U})\,.
$$
Then we have
$$
\A_X(I\cap J,\Fc)=\A_X(I,\Fc)\cap \A_X(J,\Fc)\,.
$$
\end{lemma}
\begin{proof}
Combine flatness of $\Fc$, Corollary~\ref{cor-Denis}, and left exactness of $H^0$.
\end{proof}

\begin{proof}[Proof of Theorem~\ref{theor-main}(i)]
Suppose that condition~$(i)$ is satisfied. By Lemma~\ref{lemma-red}, it is enough to consider the case $\Fc=\OO_X$. Also, assume that $0\in I$ (otherwise, there is nothing to prove). It follows from the explicit definition of adelic groups on a regular surface,~\cite[Sec. 2]{Par},~\cite[Sec. 3.3]{Osi2},~\cite[Sec. 8.5]{Mor}, that an element $a\in \A_X(I,\OO_X)$ belongs to $\A_X(I\smallsetminus\{0\},\OO_X)$ if and only if for any flag
$\Delta=(\eta_X,\eta_1,\ldots,\eta_p)$ in $M(I)$, we have that \mbox{$a_{\Delta}\in \A_X(\Delta\smallsetminus\eta_X,\OO_X)$}, where $\eta_X$ is the generic point of $X$ (if $I=\{0\}$, then the last condition should be replaced by $a\in \OO_{X,x}$ for any point $x\in X$).

Since $X$ is a regular surface, the ring $A:=\A_X(\Delta\smallsetminus\eta_X,\OO_X)$ is also regular (this follows from the explicit description of local factors in adeles on a regular surface,~\mbox{\cite[Sec. 3.3]{Osi2}}. Therefore, $a_{\Delta}\in A$ if and only if $v_D(a)\geqslant 0$ for any prime divisor~$D$ in $\Spec(A)$,~\mbox{\cite[Cor. 11.4]{Eis}}. By the definition of adeles, $v_D(a)\geqslant 0$ for all $D$ that are not analytic components of a prime divisor on $X$.
This implies that $a_{\Delta}\in A$ for all $\Delta\in M(I)$ if and only if $a\in \A_X(\{1,2\},\OO_X)$.
In other words, we have shown the equality
$$
\A_X(I\smallsetminus\{0\},\OO_X)=\A_X(I,\OO_X)\cap \A_X(\{1,2\},\OO_X)\,.
$$
Since $I\smallsetminus\{0\}=I\cap J=I\cap \{1,2\}$, the proof is finished.
\end{proof}

\section{Condition~$(ii)$ case}\label{sect:ii}

Next two lemmas are also used in Section~\ref{sect:iii}. Let $X$ and $\Fc$ be as in Theorem~\ref{theor-main}.

\begin{lemma}\label{lemma-0}
Suppose that $0\notin I\cup J$, $I,J\subset\{0,1,2\}$, and that for any flat quasicoherent sheaf $\Gc$ on $X$, we have
$$
\A_X(I\cap J,\Gc)=\A_X(I,\Gc)\cap \A_X(J,\Gc)\,.
$$
Then we have
$$
\A_X(\{0\}\cup(I\cap J),\Fc)=\A_X(\{0\}\cup I,\Fc)\cap \A_X(\{0\}\cup J,\Fc)\,.
$$
\end{lemma}
\begin{proof}
Let $\Fc_{\eta}$ denote the constant sheaf associated to the fibre of $\Fc$ at the generic point of $X$. By the definition of adeles, there is an equality (for which one does not require $\Fc$ to be coherent)
$$
\A_X(\{0\}\cup I,\Fc)=\A_X(I,\Fc_{\eta})\,.
$$
We have $\Fc_{\eta}=\varinjlim_D\Fc(D)$, where the limit is taken over all effective (not necessarily reduced) divisors $D$ on $X$. Since the functor $\A_X(I,-)$ is exact and commutes with filtered colimits, we obtain the equality
$$
\A_X(\{0\}\cup I,\Fc)=\varinjlim_D \A_X(I,\Fc(D))\,.
$$
This implies the required statement after we apply the condition of the lemma to sheaves~\mbox{$\Gc=\Fc(D)$} for various $D$.
\end{proof}

\begin{lemma}\label{lemma-reduction}
Suppose that for any flat quasicoherent sheaf $\Gc$ on $X$, we have
$$
H^0(X,\Gc)=\A_X(\{1\},\Gc)\cap\A_X(\{2\},\Gc)\,.
$$
Then for all $I,J\subset \{0,1,2\}$, we have
$$
\A_X(I\cap J,\Fc)=\A_X(I,\Fc)\cap \A_X(J,\Fc)\,.
$$
\end{lemma}
\begin{proof}
By Theorem~\ref{theor-main}$(i)$, we need to consider only pairs $(I,J)$
such that $I\cap J\ne I\smallsetminus\{0\}$.
Explicitly, it is enough to consider the following pairs $(I,J)$ (up to a permutation of~$I$ and $J$):
$$
(\{1\},\{2\}),\quad (\{0,1\},\{2\}),\quad (\{1\},\{0,2\}),\quad (\{0,1\},\{0,2\})\,.
$$
The second and the third cases are reduced to the first one as follows: one has the embedding
$$
\A_X(\{0,1\},\Fc)\cap \A_X(\{2\},\Fc)\subset \A_X(\{0,1\},\Fc)\cap \A_X(\{1,2\},\Fc)=\A_X(\{1\},\Fc)\,,
$$
where the last equality follows from Theorem~\ref{theor-main}$(i)$. Therefore, we see that
$$
\A_X(\{0,1\},\Fc)\cap \A_X(\{2\},\Fc)=\A_X(\{1\},\Fc)\cap \A_X(\{2\},\Fc)\,.
$$
The same reasoning is for the pair $(\{1\},\{02\})$.
The fourth case is reduced to the first case by Lemma~\ref{lemma-0}.
\end{proof}

\begin{rmk}\label{rmk-finite}
It follows from the proofs that Lemmas~\ref{lemma-0} and~\ref{lemma-reduction} remain true after one replaces flat quasicoherent sheaves by locally free sheaves of finite rank.
\end{rmk}

The next lemma is the only place where we use that $k$ is uncountable.

\begin{lemma}\label{lemma-school}
Let $a\in k[[u,v]]$ be a Taylor series over a field $k$. Suppose that for any polynomial \mbox{$f\in k[u,v]$} with $f(0)=0$, there is a polynomial $b_f\in k[u,v]$ such that $a\equiv b_f\pmod{f}$. Also, suppose that $k$ is uncountable. Then the series $a$ is actually a polynomial: $a\in k[u,v]$.
\end{lemma}
\begin{proof}
Let $a=\sum_{i,j\geqslant 0}a_{i,j}u^iv^j$. For each $\lambda\in k$, consider a linear polynomial $f=u-\lambda v$. Then $k[[u,v]]/(f)\simeq k[[v]]$ and $a\pmod{f}$ is the following Taylor series:
$$
\sum_{n\geqslant 0}\left(\sum_{i=0}^na_{i,n-i}\lambda^i\right)v^n.
$$
Assume that $a$ is not a polynomial. Equivalently, the sequence of polynomials
$$
p_n(t):=\sum_{i=0}^na_{i,n-i}t^i
$$
in a formal variable $t$ contains infinitely many non-zero elements. By the assumption of the lemma, for any element $\lambda$, there is $n_0$ such that for any $n\geqslant n_0$, we have $p_{n}(\lambda)=0$. This contradicts the fact that $k$ is uncountable, because each non-zero polynomial has finitely many roots and we have a countable set of polynomials $p_n$.
\end{proof}

\begin{proof}[Proof of Theorem~\ref{theor-main}(ii)]
Suppose that condition~$(ii)$ is satisfied. By Lemma~\ref{lemma-reduction}, it is enough to show the equality $H^0(X,\Fc)=\A_X(\{1\},\Fc)\cap\A_X(\{2\},\Fc)$. By Lemma~\ref{lemma-red}, we may assume that $X$ is affine and $\Fc=\OO_X$. By Noether normalization, there is a finite morphism $\pi\colon X\to \Ab^2$. By Lemma~\ref{prop-finitemorph}, the statement of the lemma for $(X,\OO_X)$ is equivalent to that for $(\Ab^2,\pi_*\OO_X)$.

Since $\Ab^2$ is regular and $X$ is equidimensional, $k[X]=(\pi_*\OO_X)(\Ab^2)$ is a projective module over $k[\Ab^2]$,~\cite[Cor. 18.17]{Eis}. Therefore, by Lemma~\ref{lemma-Denis}, we are reduced to the case $(\Ab^2,\OO_{\Ab^2})$.

For short, put $A_i:=\A_{\Ab^2}(\{i\},\OO_{\Ab^2})$, $i=1,2$. Let $a\in A_1\cap A_2$. For any irreducible curve $C\subset \Ab^2$, the restriction $a|_C:=(a_{C,x}|_C)_{x\in C}$ of $a\in \A_{\Ab^2}(\{1,2\},\OO_{\Ab^2})$ is an element in $\A_C(\{0,1\},\OO_C)$. Since $a\in A_1$, we see that $a|_C$ belongs to the field of rational functions~$k(C)$. On the other hand, we have $a\in A_2$, whence $a|_C$ is regular on~$C$ (cf. the proof of Proposition~\ref{prop-explinters}).

Let now $x\in\Ab^2$ be a $k$-point. Then the component $a_x\in \widehat{\OO}_{\Ab^2,x}$ of $a\in A_2$ satisfies the condition of Lemma~\ref{lemma-school}, whence $a_x\in \OO_{\Ab^2,x}$. Since $a\in A_1$, the rational function~$a_x$ is the same for all points $x\in X$, which finishes the proof.
\end{proof}

\section{Condition~$(iii)$ case}\label{sect:iii}

First we recall the constructions of adelic complexes and of an inverse image map on them. Given a Noetherian scheme $X$ and a quasicoherent sheaf $\Fc$ on $X$, the degree $p$ term of the corresponding adelic complex is defined by the formula,~\cite[Th. 2.4.1, Def. 5.1.1]{Hub},
$$
\A(X,\Fc)^p:=\A_X(S(X)_p,\Fc)\,.
$$
There are canonical isomorphisms,~\cite[Th. 4.2.3, Prop. 5.1.3]{Hub},
$$
H^p(X,\Fc)\cong H^p(\A(X,\Fc)^{\bullet}),\quad p\geqslant 0\,.
$$
In particular, the group of cocycles in $\A(X,\Fc)^{0}$ is equal to $H^0(X,\Fc)$.

Given a morphism of finite type $f\colon Y\to X$ between Noetherian schemes and a quasicoherent sheaf $\Fc$ on $X$, one has well-defined inverse image maps,~\cite[p.178]{Par2},
\begin{equation}\label{eq:invim}
f^*\colon \A_X(S(X)_p,\Fc)\to \A_Y(S(Y)_p,f^*\Fc),\quad p\geqslant 0\,.
\end{equation}
For any flag $\Delta\in S(Y)_p$ and $a\in \A_X(S(X)_p,\Fc)$, one has $(f^*a)_{\Delta}=f^*(a_{f(\Delta)})$, where we put~$a_{f(\Delta)}$ to be zero if $f(\Delta)$ is not a non-degenerate flag and we use otherwise the canonical map
$$
f^*\colon\A_X(\{f(\Delta)\},\Fc)\to\A_Y(\{\Delta\},f^*\Fc)\,.
$$
The maps~\eqref{eq:invim} define a morphism between the adelic complexes
$$
f^*\colon \A(X,\Fc)^{\bullet}\to \A(Y,f^*\Fc)^{\bullet}\,.
$$

\begin{prop}\label{prop-explinters}
If $\Fc$ is a locally free sheaf of finite rank on a regular surface $X$, then there is a canonical isomorphism
$$
\A_X(\{1\},\Fc)\cap \A_X(\{2\},\Fc)\cong \varprojlim_Z H^0(Z,\Fc|_Z)\,,
$$
where the projective limit is taken over all closed subschemes $Z\subset X$ (not necessarily reduced or irreducible) such that $Z\ne X$.
\end{prop}
\begin{proof}
Since $\Fc$ is coherent, there are equalities
$$
\A_X(\{1\},\Fc)=\prod_{C\subset X}\widehat{\Fc}_C,\quad \A_X(\{2\},\Fc)=\prod_{x\in X}\widehat{\Fc}_x\,,
$$
where $C$ runs though all irreducible curves on $X$, $x$ runs through all closed points on~$X$, and $\widehat{\Fc}_C$, $\widehat{\Fc}_x$ denote the completions of the stalks $\Fc_C$ and $\Fc_x$, respectively. Therefore, there is a well-defined map
$$
\varprojlim_Z H^0(Z,\Fc|_Z)\to \A_X(\{1\},\Fc)\cap \A_X(\{2\},\Fc)
$$
induced by taking limits over subschemes in $X$ whose support is either a curve $C$ or a point $x$. Let us construct an inverse map.

Given $a\in \A_X(\{1\},\Fc)\cap \A_X(\{2\},\Fc)$, consider the element
$$
\tilde a:=(0,a,a)\in \A_X(\{0\},\Fc)\oplus \A_X(\{1\},\Fc)\oplus \A_X(\{2\},\Fc)=\A(X,\Fc)^0\,.
$$
We have that
$$
d\tilde a\in \A_X(\{01\},\Fc)\oplus \A_X(\{02\},\Fc)\subset \A(X,\Fc)^1\,.
$$
Let $i\colon Z\to X$ be a closed subscheme such that $Z\ne X$.
Since $Z$ has dimension at most one, the inverse image map
$$
i^*\colon \A_X(S(X)_1,\Fc)\to \A_Z(S(Z)_1,\Fc|_Z)
$$
vanishes on both subspaces $\A_X(\{01\},\Fc)$ and $\A_X(\{02\},\Fc)$ in $\A_X(S(X)_1,\Fc)$. Therefore, $i^*(d\tilde a)=0$ and $i^*(\tilde a)$ is a cocycle in $\A(Z,\Fc|_Z)^0$. The isomorphism $H^0(Z,\Fc|_Z)\cong H^0(\A(Z,\Fc|_Z)^{\bullet})$ implies that $i^*(\tilde a)$ corresponds to a unique element $s_Z\in H^0(Z,\Fc|_Z)$. One checks that the collection $\{s_Z\}$ defines an element in the projective limit and this gives the desired map
$$
\A_X(\{1\},\Fc)\cap \A_X(\{2\},\Fc)\to \varprojlim_Z H^0(Z,\Fc|_Z)\,.
$$
This finishes the proof.
\end{proof}

\begin{proof}[Proof of Theorem~\ref{theor-main}(iii)]
Suppose that condition~$(iii)$ is satisfied. By Lemma~\ref{lemma-reduction}, it is enough to show the equality $H^0(X,\Fc)=\A_X(\{1\},\Fc)\cap\A_X(\{2\},\Fc)$. By Proposition~\ref{prop-explinters}, we need to show that the natural embedding
$$
H^0(X,\Fc)\to \varprojlim_Z H^0(Z,\Fc|_Z)
$$
is an isomorphism, where the limit is taken over all closed subschemes $Z\subset X$ (not necessarily reduced or irreducible) such that $Z\ne X$.

Since $X$ is projective, by Serre duality, there exists a very ample invertible sheaf~$\OO_X(1)$ on $X$ such that
$$
H^0(X,\Fc(-n))=H^1(X,\Fc(-n))=0
$$
for any natural number $n>0$.

Let $D$ and $E$ be the zero schemes of non-zero sections of~$\OO_X(m)$ and $\OO_X(n)$ for some natural numbers $m,n>0$. Exact sequences of cohomology groups imply that the restriction $H^0(X,\Fc)\to H^0(D,\Fc|_D)$ is an isomorphism and the restriction $H^0(E,\Fc|_E)\to H^0(D\cap E,\Fc|_{D\cap E})$ is injective, where we consider the schematic intersection $D\cap E$.

Let now $\{s_Z\}$ be an element in $\varprojlim_Z H^0(Z,\Fc|_Z)$. Fix $D$ as above and let $s\in H^0(X,\Fc)$ be the element that restricts to $s_D$ on $D$. Then for any $E$ as above, $s$ restricts to $s_E$ on~$E$, as both~$s$ and $s_E$ restrict to the same element in $H^0(D\cap E,\Fc|_{D\cap E})$.

We claim that $s$ restricts to $s_{\tilde x}$ for any closed subscheme $\tilde x$ in $X$ whose support is one closed point $x\in X$. Indeed, it is sufficient to consider the case $\tilde x=\Spec(\OO_{X,x}/\m_x^N)$, $N\geqslant 1$. There is $E$ as above that passes through $x$. Then $\tilde x$ is a closed subscheme in $N\cdot E$, which is the zero scheme of the corresponding section of $\OO_X(Nn)$. As shown above, $s$ restricts to $s_{N\cdot E}$ on $N\cdot E$, whence $s$ restricts to $s_{\tilde x}$ on $\tilde x$.

By Proposition~\ref{prop-explinters}, this implies that $\{s_Z\}$ is equal to the image of $s$ under the map $H^0(X,\Fc)\to \varprojlim_Z H^0(Z,\Fc|_Z)$.
\end{proof}

\section{Counterexample}

We provide a counterexample to Theorem~\ref{theor-main} when $X$ is affine $k$ is countable. With this aim we use the following construction.

Let $X$ be an affine surface over a field $k$. Suppose that a series $\sum_{n\geqslant 1}f_n$, $f_n\in k[X]$, converges in the complete local ring~$\widehat{\OO}_{X,\eta}$ for any schematic point $\eta\in X$ except for the generic point of $X$. Then this series defines an element $a\in\A_X(\{0,1,2\},\OO_X)$ such that for each flag \mbox{$\Delta=(X,C,x)$}, we have
$$
a_{\Delta}=\sum_{n\geqslant 1}f_n\in\widehat{\OO}_{X,C}\cap \widehat{\OO}_{X,x}\,.
$$
It follows that $a\in \A_X(\{1\},\OO_X)\cap \A_X(\{2\},\OO_X)$.

\begin{theor}\label{theor:contr}
Let $X$ be an affine regular surface over a countable field $k$. Then
$$
k[X]\subsetneq \A_X(\{1\},\OO_X)\cap \A_X(\{2\},\OO_X)\,.
$$
\end{theor}
\begin{proof}
By the construction before the theorem, it is enough to find a convergent series $\sum_{n\geqslant 1}f_n$ as above, which does not converge to an element in $k[X]$.

There are countably many prime ideals in $k[X]$ (use Hilbert's basis theorem and countability of $k[X]$). Let $\p_1,\ldots,\p_n,\ldots$ be a sequence of all non-zero prime ideals in~$k[X]$ and take the sequence of ideals
$$
\a_1:=\p_1,\quad \a_n:=\a_{n-1}^2\cdot\p_n,\,n\geqslant 2\,.
$$
Choose non-zero elements $g_n\in \a_n$, $n\geqslant 1$. Choose a closed point $x\in X$ and a non-zero function $t\in k[X]$ which vanishes at $x$. Let $(t^a)$, $a\geqslant 1$, denote the ideal $t^a\cdot\widehat{\OO}_{X,x}$. Choose an increasing exhaustive filtration of~$k[X]$ by finite-dimensional $k$-subspaces:
$$
F_1\subset F_2\subset \ldots\subset F_l\subset F_{l+1}\subset \ldots\subset k[X]\,.
$$
Since all $F_l$ are finite-dimensional and $\cap_{a\geqslant 0}(t^a)=0$, we see that for any $l$, there is a natural number $a(l)$ such that
\begin{equation}\label{eq:intersfiltr}
F_l\cap (t^{a(l)})=0\,,
\end{equation}
where the intersection is taken in $\widehat{\OO}_{X,x}$. We can also assume that $a(l)$ is strictly increasing in $l$. Put $f_1:=g_1$. Now define recursively natural numbers $l(n)$, $n\geqslant 2$, and $f_n\in k[X]$ as follows:
$$
\sum_{i=1}^{n-1}f_i\in F_{l(n)},\quad f_n:=t^{a(l(n))}\cdot g_n,\quad n\geqslant 2\,,
$$
In particular, the sequence $l(n)$ is strictly increasing in $n$ and $f_n\in F_{l(n+1)}$. The series $\sum_{n\geqslant 1}f_n$ converges in the complete local ring $\widehat{\OO}_{X,\eta}$ of any schematic point $\eta\in X$ except for the generic point of $X$. Also, this series is rather rarefied.

Denote by $f$ the sum $\sum_{n\geqslant 1}f_n$ in $\widehat{\OO}_{X,x}$. Let us show that $f$ does not belong to $k[X]\subset \widehat{\OO}_{X,x}$. Assume the converse. Since $l(n)$ is strictly increasing, we have that $f\in F_{l(n)}$ for some $n$. Since $\sum_{i=1}^{n-1}f_i\in F_{l(n)}$ and, by construction, $\sum_{i\geqslant n}f_n\in (t^{a(l(n))})$, we conclude by condition~\eqref{eq:intersfiltr} that
$\sum_{i\geqslant n}f_n=0$. On the other hand, since $f_i\in (t^{a(l(n+1))})$ for $i>n$, we see that
$$
\sum\limits_{i\geqslant n}f_i\equiv f_n\pmod{t^{a(l(n+1))}}\,.
$$
Since $0\ne f_n\in F_{l(n+1)}$, we conclude by condition~\eqref{eq:intersfiltr} that $f_n$ does not vanish modulo $(t^{a(l(n+1))})$. Therefore, $\sum_{i\geqslant n}f_i\ne 0$ and we get a contradiction.
\end{proof}

The following proposition describes the intersection $\A_X(\{1\},\OO_X)\cap \A_X(\{2\},\OO_X)$ and has interest in its own right.

\begin{prop}\label{lemma-series}
Let $X$ be an affine surface over an arbitrary field $k$.
\begin{itemize}
\item[(i)]
There is a canonical isomorphism
$$
\A_X(\{1\},\OO_X)\cap \A_X(\{2\},\OO_X)\cong \varprojlim_{\a} k[X]/\a\,,
$$
where the projective limit is taken over all non-zero ideals $\a\subset k[X]$.
\item[(ii)]
Any element in $\A_X(\{1\},\OO_X)\cap \A_X(\{2\},\OO_X)$ is obtained from a series $\sum_{n\geqslant 1}f_n$, $f_n\in k[X]$, as in the construction before Theorem~\ref{theor:contr}.
\end{itemize}
\end{prop}
\begin{proof}
First, $(i)$ is a particular case of Proposition~\ref{prop-explinters}.

Let us prove~$(ii)$. If $k$ is uncountable, then by Theorem~\ref{theor-main}$(ii)$, there is nothing to prove. Assume that $k$ is countable. Let $\a_1,\ldots,\a_n,\ldots$ be the sequence of ideals in~$k[X]$ constructed in the proof of Theorem~\ref{theor:contr}. By~$(i)$, an element $a\in \A_X(\{1\},\OO_X)\cap \A_X(\{2\},\OO_X)$ defines a compatible collection of elements $g_n\in k[X]/\a_n$. Moreover, by the construction of $\a_n$, the collection~$\{g_n\}$ defines uniquely $a$. Let $h_n\in k[X]$, $n\geqslant 1$, be any lift of $g_n$ and put
$$
f_1:=h_1,\quad f_n:=h_n-h_{n-1},\,n\geqslant 2\,.
$$
Then the series $\sum_{n\geqslant 1}f_n$ converges in the complete local ring $\widehat{\OO}_{X,\eta}$ of any schematic point $\eta\in X$ except for the generic point and corresponds to $a$, because for any $m\geqslant 1$, we have that $f_{m+1}\in \a_m$ and $\sum_{n=1}^{m}f_n\equiv g_m\pmod{\a_m}$.
\end{proof}

\end{document}